\documentclass[preprint,12pt]{elsarticle}

\usepackage{amssymb}
\usepackage{amsmath}
\usepackage{epsfig}
\usepackage{amsthm}
\usepackage{euscript}
\usepackage{latexsym, mathrsfs}

\newtheorem{theorem}{Theorem}[section]
\newtheorem{lemma}[theorem]{Lemma}

\theoremstyle{definition}
\newtheorem{remark}[theorem]{Remark}
\newtheorem{defn}[theorem]{Definition}

\newtheorem{question}[theorem]{Question}

\numberwithin{equation}{section}
\def \mff{\mathsf}

\def \mc{\mathcal}
\def \inv{^{-1}}

\def \v{\vskip 0.1in}
\def \n{\noindent}

\def \real{\mathbb{R}}

\def \p{\partial}

\begin{document}

\begin{frontmatter}

\title{Interior Regularity for a generalized Abreu Equation \tnoteref{label1}}
\tnotetext[label1]{Li acknowledges the support of NSFC Grants NSFC11521061.
Sheng acknowledges the support of NSFC Grants NSFC11471225.}

%% use optional labels to link authors explicitly to addresses:
%% \author[label1,label2]{<author name>}
%% \address[label1]{<address>}
%% \address[label2]{<address>}
\author[Li]{An-Min Li}
\address[Li]{Yangtze Center of Mathematics Department of Mathematics
Sichuan University Chengdu, 610064, China}

\ead{anminliscu@126.com}
\author[Li]{Zhao Lian}
\address[Li]{Department of Mathematics Sichuan University
Chengdu, 610064, China}
\ead{zhaolian.math@gmail.com}

\author[Sheng]{Li Sheng\corref{cor1}}
\address[Sheng]{Department of Mathematics Sichuan University Chengdu, 610064, China}
\ead{lshengscu@gmail.com}
\cortext[cor1]{Corresponding author}
\begin{abstract}
We study a generalized Abreu Equation in $n$-dimensional polytopes and derive interior estimates of solutions under the assumption of the
uniform $K$-stability.
\end{abstract}

\begin{keyword}
%% keywords here, in the form: keyword \sep keyword
Interior estimates\sep  generalized Abreu Equation.
%% MSC codes here, in the form: \MSC code \sep code
%% or \MSC[2008] code \sep code (2000 is the default)
\MSC[2008] 53C55 \sep 35J60
\end{keyword}

\end{frontmatter}

%%
%% Start line numbering here if you want
%%
% \linenumbers

%% main text

\section{Introduction}\label{Sec-Intro}
%\v\n
The existence of extremal and contant scalar curvature is a central problem in K\"ahler geometry. In a series of papers \cite{D1}, \cite{D2}, \cite{D3}, and \cite{D4}, Donaldson studied
this problem on toric manifolds and proved the existence
of metrics of constant scaler curvatures on toric surfaces
under an appropriate stability condition. Later on in \cite{CLS1} and \cite{CLS2}, Chen, Li and Sheng proved the existence of metrics of prescribed scaler curvatures on toric surfaces under the uniform stability condition.
\v
It is important to generalize the results of Chen, Li and Sheng to more general K\"ahler manifold. This is one of a sequence of papers, aiming at generalizing the results of Chen, Li and Sheng to homogeneous toric bundles. The primary goal of this paper is to study the following nonlinear fourth-order partial differential equation
for an $n$-dimensional convex function $u$
\begin{equation}\label{eqn 1.1}
\frac{1}{\mathbb{D}}\sum_{i,j=1}^n\frac{\partial^2 \mathbb{D}u^{ij}}{\partial \xi_i\partial \xi_j}=-A.
\end{equation}
Here, $\mathbb{D}>0$ and $A$ are two given smooth functions on $\bar{\Delta}$ and $(u^{ij})$ is the inverse of the Hessian matrix $(u_{ij})$. The equation \eqref{eqn 1.1} was introduced by Donaldson \cite{D5}
in the study of the scalar curvature of toric fibration, see also \cite{R} and \cite{N-1}. In \cite{LSZ} the authors also derived this PDE in the study of the scalar curvature of homogeneous toric bundles. We call \eqref{eqn 1.1} a generalized Abreu Equation. The main result is the following interior estimate
\begin{theorem}\label{theorem_1.3}
Let $\Delta$ be a bounded open polytope in $\real^n$
and $\mathbb{D}>0$, $A$ be two smooth functions on $\bar\Delta$.
Suppose  $(\Delta, \mathbb{D}, A)$ is  uniformly $K$-stable and $u$ is a solution in $\mathbf{S}_{p_o}$
of  the equation \eqref{eqn 1.1}.
Then, for any $\Omega\subset\subset \Delta$, any nonnegative integer $k$ and any constant $\alpha\in (0,1)$,
\begin{equation*}
\|u\|_{C^{k+3,\alpha}(\Omega)}\leq C,
%\nonumber
\end{equation*}
where $C$ is a positive constant depending only on $n$, $k$, $\alpha$,
$\Omega$, $\mathbb{D}$, $\|A\|_{C^k(\bar\Delta)}$ and $\lambda$ in the uniform $K$-stability.
\end{theorem}
A equivalent statement of Theorem \ref{theorem_1.3} is the following
\begin{theorem}\label{theorem_1.2}
Suppose that $(\Delta, \mathbb{D}, A)$ is  uniformly $K$-stable and that
$\{A^{(k)}\}$ is a sequence of smooth functions in $\bar\Delta$ such that
$A^{(k)}$ converges to $A$ smoothly in $\bar \Delta$.
Assume $u^{(k)}\in \mathbf{S}_{p_o}$ is a sequence of solutions of
the generalized Abreu Equation
\begin{equation}\label{eqn 1.6}
\sum_{i,j}\frac{\partial^2(\mathbb{D}u^{(k)ij})}{\partial\xi_i\partial\xi_j}=-A^{(k)}\mathbb{D}\quad\text{in }\Delta.
\end{equation}
Then there is a subsequence, still denoted by $u^{(k)}$, such that
$u^{(k)}$ converges smoothly, in any compact set $\Omega\subset \Delta$,
to some smooth and strictly convex function  $u$ in $\Delta$.
\end{theorem}
The main ideal of the proof is following:
\v
Note that, as Donaldson pointed out that, the uniform stability of $(\Delta, \mathbb{D}, A)$ implies
that there is a subsequence, still denoted by $u^{(k)}$, locally uniformly converging to $u$ in $\Delta$. The key point is to prove that $u$ is smooth and strictly convex. We consider the Legendre transform
$f^{(k)}$ of $u^{(k)}.$
Then $f^{(k)}$ satisfy the PDE
\begin{equation}\label{eqn 1.3}
-\sum_{i,j} f^{ij}\frac{\p^2(\log \mathbb F)}{\p x_i \p x_j} -\sum_{i,j}f^{ij} \frac{\p (\log \mathbb F)}{\p x_i}\frac{\p (\log \mathbb D)}{\p x_j} = A.\end{equation}
In Section \ref{Sec-Determinants}, we derive an uniform lower bound and an uniform upper bound of the determinants of the Hessian of $f^{(k)}$. We can not directly apply the  Caffarelli and Guti\'{e}rrez theory to the PDE \eqref{eqn 1.3}.
We prove a convergence theorem for this PDE in Section \ref{Sec-Convergence}. Then Theorem \ref{theorem_1.2}
follows.

\section{Uniform stability}\label{Sec-Uniform}

Let $\Delta$ be a Delzant polytope in $\real^n$, $c_k$ be a constant and
$h_k$ be an affine linear function in $\mathbb R^n$, $k=1, \cdots, d$.
Suppose that
$\Delta$ is defined by linear inequalities $h_k(\xi)-c_k>0$,  for $k=1, \cdots, d$,
where each $h_k(\xi)-c_k=0$ defines a facet of $\Delta$.
%({\bf Is $H_k$ a linear function?})
Write $\delta_k(\xi)=h_k(\xi)-c_k$
and set
\begin{equation}\label{eqn2.1}
v(\xi)=\sum_k\delta_k(\xi)\log\delta_k(\xi).
\end{equation}
This function was first introduced by Guillemin \cite{Guillemin1994}.
It defines a K\"ahler metric on the toric variety defined by $\Delta$.
We introduce several classes of functions. Set
\begin{align*}
\mc C&=\{u\in C(\bar\Delta):\, \text{$u$ is
convex on $\bar\Delta$ and smooth on $\Delta$}\},\\
\mathbf{S}&=\{u\in C(\bar\Delta):\, \text{$u$ is convex  on $\bar\Delta$
and $u-v$ is smooth on $\bar\Delta$}\},\end{align*}
where $v$ is given in \eqref{eqn2.1}.
For a fixed
point $p_o\in \Delta$, we consider
\begin{align*}
{\mc C}_{p_o}&=\{u\in \mc C:\, u\geq u(p_o)=0\},\\
\mathbf{S}_{p_o}&=\{ u\in \mathbf{S} :\, u\geq u(p_o)=0\}.\end{align*}
We say functions in ${\mc C}_{p_o}$ and ${\mathbf{S}}_{p_o}$ are {\it normalized} at $p_o$. Let
\begin{eqnarray*}
\mc C_\ast&=&\{u| \mbox{there exist a constant $C>0$  and  a  sequence of $\{u^{(k)}\}$ in ${\mc C}_{p_o}$ }\\
&&\mbox{such that
   $\int_{\partial\Delta}u^{(k)} \mathbb{D}d\sigma<C$ and
$u^{(k)}$  locally uniformly converges to} \\
&& \mbox{$u$ in $\Delta$}\}.
\end{eqnarray*}
For any $u\in \mc C_\ast,$   define $u$ on boundary as
$$u(q)=\lim_{\Delta\ni \xi\to q} u,\;\;\; q\in \partial \Delta.$$
Let $P>0$ be a  constant, we define
$$
\mc C_\ast^P=\{u \in\mc C_\ast| \int_{\partial\Delta}u \mathbb{D}d\sigma\leq P  \}.
$$

\v
Following \cite{N-1} we consider the  functional
\begin{equation}\label{eqn 2.2}
\mc F_A(u)=-\int_\Delta \log\det(u_{ij})\mathbb{D}d \mu+\mc L_A(u),
\end{equation}
where
\begin{equation}\label{eqn 2.3}
\mc L_A(u)=\int_{\partial\Delta}u \mathbb{D}d\sigma-\int_\Delta Au \mathbb{D} d\mu.
\end{equation}
$\mc F_A$ is called the Mabuchi functional
and $\mc L_A$ is closely related to the Futaki invariants. The Euler-Lagrangian equation for $\mc F_A$ is \eqref{eqn 1.1}.
It is known that,
if $u\in \mathbf{S}$ satisfies the equation \eqref{eqn 1.1}, then $u$ is an absolute minimizer for
$\mc F_A$ on $\mathbf{S}$.

\begin{defn}\label{defn_1.5}
Let $\mathbb{D}>0$ and $A$ be two smooth functions on $\bar\Delta$.
Then, $({\Delta},\mathbb{D},A)$ is called {\em uniformly $K$-stable} if
the functional $\mc L_A$ vanishes on affine-linear functions and
there exists a constant $\lambda>0$
such that, for any $u\in  {\mc C}_{p_o}$,
\begin{equation}\label{eqn 1.6}
\mc L_A(u)\geq \lambda\int_{\partial \Delta} u \mathbb{D}d \sigma.
\end{equation}
We also say that $\Delta$ is
$(\mathbb{D}, A,\lambda)$-stable.
\end{defn}
\begin{remark}
The conditions in Definition \ref{defn_1.5} are exactly the contents of Condition 1
\cite{N-1}. Following Donaldson we call it the {\em uniform $K$-stability}.
\end{remark}
Using the same method in \cite{CLS4} we immediately get
\begin{theorem}\label{theorem_2.3}
If the equation \eqref{eqn 1.1} has a solution in $\mathbf{S}$, then $(\Delta, \mathbb{D}, A)$ is uniform K-stable.
\end{theorem}
Namely, the uniform K-stability is a necessary condition for existing a solution of \eqref{eqn 1.1} in $\mathbf{S}$. We pose the
\begin{question}\label{question_1.8}
Let $\Delta \subset \mathbb{R}^{n}$ be a Delzant polytope, $\mathbb{D}>0$ and $A$ be two smooth functions on $\bar\Delta$.
Does the uniform K-stability of $(\Delta, \mathbb{D}, A)$ imply that the equation \eqref{eqn 1.1} has a solution in $\mathbf{S}$?
\end{question}

Assume that $v\in \mathbf{S}_{p_o}$ is the solution of the equation \eqref{eqn 1.1}, and $u$ is a convex function. For any segment $I\subset\subset \Delta$, $u$ defines a convex function $w:=u|_I$ on $I$. It defines a Monge-Ampere measure on $I$, we denote this by $N$. The key point of the proof in \cite{CLS4} is the following lemma.
\begin{lemma}\label{lemma_2.5}
Let $u\in \mc C^P_\ast$  and $u^{(k)}\in \mathcal C$ locally uniformly converges to $u.$ If $N(I)=m>0$, then $$\mc L_A(u^{(k)})> \tau m$$ for some positive constant $\tau$ independent of k.
\end{lemma}
In our present case this lemma still holds due to $C^{-1}\leq \mathbb{D}\leq C$ for some constant $C>0$. For reader's convenience we give the proofs here.
\v\n
{\bf Proof of Lemma \ref{lemma_2.5}.} Let $p$ be the midpoint of $I$.  We choose coordinate system $\{0,\xi\}$ such that
$p$ is the origin, $I$ is on the $\xi_1$ axis and  $I=(-a,a)$. Set $I_{\epsilon}=[-a+\epsilon,a-\epsilon].$ By choosing $\epsilon$ small we can assume that
\begin{equation}\label{eqn 2.7}
N(I_{\epsilon})\geq \frac{3m}{4}.
\end{equation}
 Suppose that there is a Euclidean ball $B:=B_{\epsilon_o}(0)$ in $\xi_1=0$ plane
such that
$I\times B\subset\subset \Delta$. Suppose that $u$ is a limit of a sequence $u^{(k)}\in \mc C$.
Then $u^{(k)}$ converges to $u$ uniformly on $I\times B$. We have
\begin{equation}\label{eqn_5.1a}
\mc L_A(u^{(k)})= \int_{\Delta} v^{ij} u^{(k)}_{ij}\mathbb{D}d\mu.
\end{equation}
Consider the functions
$$
w^{(k)}_{\xi}(\xi_1)=u^{(k)}(\xi_1,\xi),\;\;\; \xi_1\in I, \xi\in B.
$$
We denote by
$N^{(k)}_{\xi}$ the Monge-Ampere measure on $I$ induced by $w^{(k)}_{\xi}$ . We claim that
there exists a small $B$ and large $K$ such that for any $\xi\in B$, $k>K$
\begin{equation}
N^{(k)}_{\xi}(I)\geq m/2.
\end{equation}
In fact, if not, then there exists a subsequence of $k$, still denote by $k$, and a sequence of
$\xi_k\in B$ with $\xi_k\to 0$ such that $N^{(k)}_{\xi_k}(I)< m/2$. However, by the weakly convergence of Monge-Ampere measure, we have
  $$N(I_{\epsilon} )\leq \lim_{k\to\infty}N^{(k)}_{\xi_k}(I)\leq m/2,$$ this contradicts \eqref{eqn 2.7}.
\v
On the other hand, the eigenvalues of $v^{ij}$ are bounded below in $I\times B$, let $\delta$ be the lower bound. Then
\begin{eqnarray*}
\mc L_A(u^{(k)})
 &\geq& \int_{I\times B} v^{ij}u^{(k)}_{ij}\mathbb{D}d\mu
\geq \frac{\delta}{C} \int_{I\times B} Trace(u^{(k)}_{ij})d\mu \\
&\geq& \frac{\delta}{C}\int_{I\times B} u^{(k)}_{11}d\mu=\frac{\delta}{C}\int_B N^{(k)}_{\xi}(I)d\xi
\geq  \frac{m\delta}{2C} Vol(B).
\end{eqnarray*}
This completes the proof of Lemma \ref{lemma_2.5}.
\v\n
Then by the same method as in \cite{CLS4} we can prove Theorem \ref{theorem_2.3}.

\section{Estimates of the Determinant}\label{Sec-Determinants}

Set
\begin{equation}\label{eqn 3.1}
\mathbb{F}:=\frac{\mathbb D}{\det(u_{ij})},\;\;\;U^{ij}= \det(u_{kl})u^{ij}.\end{equation}
Since $\sum_{i} U^{ij}_{i}=0,$ the generalized Abreu Equation \eqref{eqn 1.1} can be written in terms of $(\xi, u)$ as
\begin{equation}\label{eqn 3.2}
-\sum_{i,j} U^{ij}\frac{\p^2\mathbb F}{\p \xi_i \p \xi_j}  = A\mathbb D.
\end{equation}
Through the normal map $\nabla u$ we can view $\mathbb{D}$ as function in $x$. In terms of $(x, f)$ the PDE \eqref{eqn 3.2} can be written as
\begin{equation}\label{eqn 2.1}
-\sum_{i,j} f^{ij}\frac{\p^2(\log \mathbb F)}{\p x_i \p x_j} -\sum_{i,j}f^{ij} \frac{\p (\log \mathbb F)}{\p x_i}\frac{\p (\log \mathbb D)}{\p x_j} = A.
\end{equation}

\subsection{\bf The lower bound of the determinant}\label{lower bound}
\v
\label{sect_3.1}
The following Lemma is proved in \cite{N-1} for toricfibration. It can be extend directly to the generalized Abreu Equation \eqref{eqn 1.1}.
\begin{lemma}\label{lemma_3.1} Let $\Delta$ be a bounded open polytope in $\real^n$ and $\mathbb{D}>0$, $A$ be two smooth functions on $\bar\Delta$. Let $u\in \mathcal{C}$ be a strictly convex function satisfying the generalized Abreu Equation \eqref{eqn 1.1}. Suppose that $\mathbb{F}=0$ on $\partial \Delta$. Then
$$\det ( u_{ij})\geq \mff C_1(sup_{\Delta} A)^{-n}$$ everywhere in $\Delta$,
where $\mff C_1$ is a constant depending on $n$, $\mathbb{D}$ and $\Delta$.
\end{lemma}

In the following we derive a more stronger estimate than Lemma \ref{lemma_3.1}, which will be used in our next papers.  First we prove a  preliminary lemma.
\begin{lemma}\label{lemma 3.2}
Let $\Delta$ be a bounded open polytope.
Suppose that $\mathbb{F}=0$ on $\partial \Delta$. Let $E$ be an edge of
$\Delta$. Suppose that $E$ is given by $\xi_1=0$.  Set $$v(\alpha,\beta,C) =- \xi_1^{\alpha}(C-\xi_1)^{\beta}\left(C - \sum_{j=2}^{n}
\xi_j^2\right)^{\beta},$$ where $\alpha,\beta,C$ are constants. Then   for any $ \frac{1}{2n}\leq\alpha,\beta \leq 1-\frac{1}{2n}$, there exists  constants $C,C_{1}>0$ depending only on $n$ and $diam(\Delta)$ such that
$v$ is strictly convex and
\begin{equation}
\det(v_{ij})>C_{1}(\epsilon_0) \xi_1^{n\alpha-2}.
\end{equation}
\end{lemma}
\begin{proof} Choose  $C>0$  large such
that
\begin{equation}\label{eqn 3.5}
\Delta \subset \left \{\xi \;|\xi_1\leq
\frac{C}{m}\right\}\bigcap \left\{\xi\;|\sum_{j=2}^{n} \xi_j^2 \leq
\frac{C}{m}\right\},
 \end{equation}
 where $m=8n$.   We calculate $det(v_{ij})$. For any point
$\xi$, By taking an orthogonal transformation of $\xi_2,...,\xi_n$,
we may assume that $\xi = (\xi_1,\xi_2,0,...,0)$. By a direct
calculation we have
$$v_{11}= -v\left[-\left(\frac{\alpha }{\xi_{1} } -
\frac{\beta}{C-\xi_{1}}\right)^2+\frac{\alpha}{\xi_{1}^{2}} +
\frac{\beta}{(C-\xi_{1})^{2}} \right],$$
$$v_{12} =-v \left(\frac{\alpha }{\xi_{1} } -
\frac{\beta}{C-\xi_{1}}\right)\frac{2\beta
\xi_2}{ C-\xi_2^2  },\;\;\;\;\;v_{ij}= 0,\;\;i>2, \;\;i\ne j.$$$$v_{22}=-v\left[\frac{2\beta(C+\xi_2^2)}{(C-\xi_2^2)^{2}}
- \frac{4\beta^2 \xi_2^2}{(C-\xi_2^2)^{2}}\right],\;\;\;\;v_{ii}= -v\frac{2\beta}{ C-\xi_2^2 } ,\;\;\;i>2.$$ Denote $A-B= v_{11}v_{22}-v_{12}^2 ,\;D=\prod_{i=3}^{n} v_{ii}.$ The determinant of
$(v_{ij})$ is $\det(v_{ij}) =(A-B)\cdot D.$ A direct calculation gives us
\begin{align*}
A-B=&\frac{2\beta v^2}{\xi^2_{1} (C-\xi_{1})^2(C-\xi_{2}^2)^2}\left[ \alpha(C-\xi_{1})^2((1-\alpha)C+ (1-2\beta- \alpha)\xi_{2}^2)\right.\\
 &\left.+ \beta \xi_{1}^2 ((1-\beta)C+ (1-3\beta)\xi_{2}^2) +2\alpha \beta\xi_{1}(C-\xi_{1})(C+\xi_{2}^2)  \right] \\
D=&\prod_{i=3}^{n} v_{ii}=\left[-v\frac{2\beta}{ C-\xi_2^2 }\right]^{n-2}.
\end{align*}
 For any $\alpha,\beta$ satisfy $\frac{1}{2n}\leq \alpha,\beta\leq
1-\frac{1}{2n}$, by $m>4(2n-1),$ we have
\begin{equation}\label{3.3}
A-B\geq \frac{ \alpha\beta v^2}{\xi^2_{1} (C-\xi_{2}^2)^2}\frac{C(2n-1)}{2nm} .
\end{equation}
It is easy to check that $v$ is strictly convex and
\begin{equation}
\det(v_{ij})>C(n) \xi_1^{n\alpha-2}.
\end{equation}
\end{proof}

Now we prove
  \noindent \vskip
0.1in \noindent
\begin{lemma} Let $\Delta$ be a bounded open polytope in $\real^n$
and $\mathbb{D}>0$, $A$ be two smooth functions on $\bar\Delta$.
Let $u\in \mathcal{C}$ be a strictly convex function satisfying the generalized Abreu Equation \eqref{eqn 1.1}.
Suppose that $\mathbb{F}=0$ on $\partial \Delta$. Let $E$ be an edge of
$\Delta$. Suppose that $E$ is given by $\xi_1=0$. Let $p\in E^o$.
Then the following estimate holds in a neighborhood of $p$
$$ det(u_{ij})\geq \frac{b}{\xi_1}$$
for some constant $b>0$ depending only on $n$, $diam(\Delta)$, $\max_{\bar\Delta}\mathbb D$, $\min_{\bar\Delta}\mathbb D$ and $\|A\|_{L^{\infty}(\Delta)}$.
\end{lemma}
\begin{proof} First we prove that there exists a constant $b_{0}>0$ such that
\begin{equation}\label{eqn 3.7}
det(u_{ij})\geq
b_0\xi_1^{-(1-\frac{1}{n}) }.\end{equation}
Choose    $\beta=\frac{1}{2}$ in Lemma \ref{lemma 3.2}. Let $C>0$ and $m=8n$ be constants such
that \eqref{eqn 3.5} holds. We discuss two cases. \v\n
{\bf Case 1.} $n=2.$ We choose $\alpha=\frac{1}{2}$ and consider
the following function
$$h= \mathbb F + b_1v.$$
Obviously, $h<0$ on $\partial \Delta$. We have
\begin{eqnarray*}
\sum U^{ij}h_{ij}& =& -A\mathbb D  + b_1 \det(u_{ij})\sum u^{ij}v_{ij}
\\
&\geq&
-A\mathbb D + nb_1 \det(u_{ij})^{1-1/2}(\det(v_{ij}))^{1/2}\\
&\geq& -A\mathbb D +nb_1d_1C(\epsilon_0).
\end{eqnarray*}
Here we use the estimate $\det(D^2u)\geq d_1$. By  choosing the
constant $b_1$ large, we have $\sum U^{ij}h_{ij} \geq 0$. So $h$
attains its maximum on $\partial \Delta$. Then $w \leq b_1 \mathbb D^{-1} |v|.$ It
follows that
$$\det(u_{ij})\geq b_2\xi_1^{\frac{-1}{2}}.$$
for some constant $b_2$.
\v\n
{\bf Case 2. } $n\geq 3$. Choose a sequence $\{\alpha_{k}\}$ such that
$$
 \alpha_{k}=2\left(1-(1-\tfrac{1}{n})^{k}\right),\;\;\;\;\;\;\forall k\geq 1.
$$
Obviously,
\begin{equation}\label{eqn 3.8}
\alpha_{k}-\frac{2}{n}=(1-\frac{1}{n})\alpha_{k-1},\;k\geq 2,
\end{equation}
and there is $k^\star\in \mathbb Z^{+}$ such that $\alpha_{k^\star}<1-\frac{1}{n}$ and $\alpha_{k^\star+1}\geq 1-\frac{1}{n}.$
\v\n
We first let $\alpha=\alpha_{1}$, $h= \mathbb F + b_2v$. By the same argument as in Case 1 we get
$$
\det(u_{ij})\geq b_2'\xi_1^{\frac{-2}{n}}.
$$
Next we let $\alpha=\alpha_{2}$, $h= \mathbb F + b_3v$. Then
\begin{eqnarray*}
\sum U^{ij}h_{ij}& \geq &  -A\mathbb D + nb_3 \det(u_{ij})^{1-1/n}(\det(v_{ij}))^{1/n}\\
&\geq& -A\mathbb D +nb_3 b_{2'}^{1-\frac{1}{n}}\xi_1^{-\alpha_1(1-\frac{1}{n})+\alpha_{2}-\frac{2}{n}}   \geq -A\mathbb D +nb_3 b_{2'}^{1-\frac{1}{n}}.
\end{eqnarray*}
We choose $b_3$ such that $\sum U^{ij}h_{ij}>0$. Then we have
$$
\det(u_{ij})\geq b_3'\xi_1^{-\alpha_2}.
$$
We iterate the process to improve the estimate. After finite many steps we get $\det(u_{ij})\geq b'\xi_1^{-\alpha_{k^\star}}.$ Then we set $\alpha=1-\tfrac{1}{n},$ and repeat the argument above to get
\eqref{eqn 3.7}.

\v
  Next we consider the function
$$v'= \xi_1^{\alpha}\left(C + \sum_{j=2}^{n}
\xi_j^2\right) - a \xi_1,$$ where $a>0$, $\alpha >1$ are constants,
$C>0$ is the constant as before. We choose $a$ large such that
$v'\leq 0$ on $\Delta$. For any point $\xi$ we may assume that $\xi
= (\xi_1,\xi_2,0,...,0)$. By a direct calculation we have
$$v'_{11}= \alpha (\alpha-1)\xi_1^{\alpha-2}(C+
\xi_2^2),$$
$$v'_{ii}= 2\xi_1^{\alpha}\;\;\;i\geq 2,\;\;\;\;v'_{12}= 2\alpha \xi_2\xi_1^{\alpha-1},$$
$$det(D^2v')=
2^{n-1}\left[\alpha (\alpha-1)(C+ \xi_2^2) -
2\alpha^2\xi_2^2\right]\xi_1^{n\alpha-2}.$$ Then for large $C$, we
conclude that $v'$ is convex and
\begin{equation}
det(D^2v')\geq C_1\xi_1^{n\alpha-2}.
\end{equation}
Set $\alpha=1+\frac{1}{n^2}$.
 Consider
the function
$$h' = \mathbb F+ b_5v'.$$
Obviously, $h' < 0$ on $\partial \Delta$. We have
\begin{eqnarray*}
\sum U^{ij}h_{ij}& =& -A\mathbb D  + b_5\det(u_{ij})\sum u^{ij}v'_{ij}\\
& \geq&
-A\mathbb D  + nb_5 \det(u_{ij})^{1-1/n}\det(v'_{ij})^{1/n}\\
&\geq& -A\mathbb D  + nb_5
C(n)\xi_1^{-(1-\frac{1}{n})^2}C_1\xi_1^{\alpha-\frac{2}{n}}
\\
&=&-A\mathbb D +nb_5C(n)C_1.
\end{eqnarray*}
We choose $b_5$ such that
$\sum U^{ij}h_{ij}\geq 0$. By the maximum principle we have $w \leq C_5 \mathbb D^{-1} |v'|\leq aC_5\xi_1$. It
follows that % for any facet $F$,
 $\det(u_{ij})(\xi)\geq {\mff C_5}{\xi_1}\inv $ for some
constant $\mff C_5>0$ independent of $p.$
\end{proof}

\subsection{\bf  The upper bound of the determinant }\label{sect_4.2}
\v

Let $u\in \mathbf{S}_{p_o}$ be a solution of the generalized Abreu Equation \eqref{eqn 1.1}. In this section, we derive a global upper bound of the determinant of the
Hessian of $u$.
The proof of the following lemma is standard
\begin{lemma}\label{lemma_3.3}
Suppose that $u\in \mathbf{S}_{p_o}$ satisfies the generalized Abreu Equation \eqref{eqn 1.1}. Assume
that the section
$$\bar{S}_u(p_o,C)=\{\xi\in \Delta:\, u(\xi)\leq C\}$$
is compact and that there is a constant $b>0$ such that $$\sum
_{k=1}^n \left(\frac{\partial u}{\partial \xi_k}\right)^2 \leq b
\quad\text{on }\bar{S}_u(p_o,C).$$ Then,
$$\det (u_{ij})\leq \mff C_2\quad\text{in }S_u(p_o,C/2),$$
where $ \mff C_2$ is a positive constant depending on
$n$, $C$ and $b$.
\end{lemma}

Following \cite{CHLS} we derive a global estimate for the upper bound of $\det (u_{ij})$ for the generalized Abreu Equation \eqref{eqn 1.1}.
This upper bound relates to the Legendre transforms of solutions.

For any point $p$ on $\partial \Delta$, there is an affine coordinate
$\{\xi_1,..., \xi_n\}$, such that, for some $1\leq m \leq n$,
a neighborhood $U\subset \bar\Delta$ of $p$  is defined
by $m$ inequalities
$$\xi_1\ge 0,\quad ...,\quad \xi_m\geq 0,$$
with  $\xi(p)=0.$   Then, $v$ in \eqref{eqn2.1} has the form
$$v=\sum_{i=1}^{m}\xi_i\log \xi_i+\alpha(\xi),$$
where $\alpha$ is a smooth function in $\bar U$.
By Proposition 2 in \cite{D2},  we have the following result.

\begin{lemma}\label{lemma_3.4} There holds
$$\det(v_{ij})= \big[\xi_1\xi_2 ...\xi_m \beta(\xi)\big]^{-1}\quad\text{in }\Delta,$$
where $\beta(\xi)$ is smooth up to the boundary and $\beta(0)=1$.
\end{lemma}

For any $q\in \Delta$ denote by $d_E(q,\partial \Delta)$ the Euclidean distance from $q$ to $\partial \Delta$.
By Lemma \ref{lemma_3.4}, we have
\begin{equation}\label{eqn 3.2}
\det(v_{ij})\leq \frac{C}{[d_E(p,\partial \Delta)]^n}\quad\text{in  } \Delta,
\end{equation}
where $C$ is a positive constant.
\v

Recall that $p_o\in\Delta$ is the point we fixed for $\mc S_{p_o}$.
Now we choose coordinates $\xi_1,...,\xi_n$ such that $\xi(p_o)=0$. Set
$$x_i=\frac{\partial u}{\partial \xi_i},\;\;\; f=\sum_i x_i\xi_i - u.$$

\begin{lemma}\label{lemma_3.5} Let $\Delta$ be a bounded open polytope in $\real^n$ and $\mathbb{D}>0$, $A$ be smooth functions on $\bar\Delta$.
Let $u\in \mathbf{S}_{p_o}$ be a strictly convex function satisfying the generalized Abreu Equation \eqref{eqn 1.1}.
Assume, for some positive constants $d$ and $b$,
$$\frac{1+\sum x_i^2}{(d + f)^2}\leq b\quad\text{in }\mathbb R^n.$$
Then,
$$\exp\left\{ -\mff C_3 f \right\}\frac{\det (u_{ij})}{\left(d+f\right)^{2n}}\leq \mff C_4
\quad\text{in }\Delta,$$
where $\mff C_3$ is a positive constant depending only on $n$ and $\Delta$,
and $\mff C_4$ is a positive constant depending only on
$n$, $d$, $b$, $\mathbb{D}$ and $\max_{ \bar\Delta}|A|$.
\end{lemma}

\begin{proof}
Let $v$ be given as in \eqref{eqn2.1}.
By adding a linear function, we assume that $v$ is also normalized at $p_o$.
Denote $g=L(v)$. By \eqref{eqn 3.2}, it is straightforward to check that
there exists a positive constant $C_1$ such that
$$\det(v_{ij})e^{-C_1g}\to 0\quad\text{as }p\to \partial \Delta.$$
Since $u=v+\phi $ for some $\phi\in C^\infty(\bar \Delta)$, then
\begin{equation}\label{eqn3.a}\det(u_{ij})e^{-C_1 f }\to 0\quad\text{as }p\to \partial \Delta.
\end{equation}
Consider the function for some constant $\varepsilon$ to be determined,
$$\mathcal{F}=\exp\left\{ -C_1 f + \varepsilon\frac{1+\sum x_i^2}{(d + f)^2}\right\}
\frac{1}{\mathbb F\left(d+f\right)^{2n}},$$
where   $\mathbb{F}$ is defined in \eqref{eqn 3.1}, $\varepsilon$ is a positive number to be determined latter.
Obviously, $\mathcal{F}\to 0$ as $p\in\partial\Delta$.
Assume $\mathcal{F}$ attains its maximum at an interior point $p^*$.
Then at $p^*$, we have
$$
\frac{\p}{\p x_{j}}  \mathcal{F}=0,\;\;\;\;\;\;\;\sum f^{ij}\frac{\p^2  \mathcal{F}}{\p x_{i}\p x_{j}}\leq 0.
$$
Thus,
\begin{equation}\label{eqn 2.10}
-(log\mathbb F)_i - C_1f_{i}- \frac{2n f_{i}}{d+f}
+ \varepsilon\frac{1+\sum x_k^2}{(d + f)^2}\left[\frac{(\sum x_k^2)_{i}}{1+\sum x_k^2}-
2\frac{f_{i}}{d + f}\right]=0,\end{equation}
and
\begin{align}\label{eqn3.4} \begin{split}
&  \sum_{i,j}f^{ij} (\log \mathbb D)_{i}(\log \mathbb F)_{j} + A- C_1n - \frac{2n^2}{d+f}
+  \frac{2n\|\nabla f\|^2}{(d+f)^2}\\
&\quad+\varepsilon\frac{1+\sum x_k^2}{(d + f)^2}
\bigg[\frac{2\sum_{k} f^{kk}}{1+\sum x_k^2}-\frac{\|\nabla\sum x_k^2\|^2}{(1+\sum x_k^2)^2}
-\frac{2n}{d + f} + \frac{2\|\nabla f\|^2}{(d + f)^2}\bigg]\\
&\quad+\varepsilon\frac{1+\sum x_k^2}{(d + f)^2}\left\|\left(\frac{\nabla(\sum x_k^2)}{1+\sum x_k^2}-
\frac{2\nabla f}{d + f}\right)\right\|^2
\leq 0, \end{split}
\end{align}
where we used \eqref{eqn 2.1} and denote $F_{i}=\frac{\p F }{\p x^{i}},F_{ij}=\frac{\p^2 F }{\p x^{i} \p x^{j}}$ for any function $F$.
Since $$\sum \left| \tfrac{\p \log \mathbb D}{\p \xi_{i}} \right|\leq C,$$ and
\begin{equation*}
\sum_{i,j}f^{ij}  \frac{\p\log \mathbb D}{\p x_{i}} \frac{\p \log \mathbb F}{\p x_{j}}=\sum_{i,j,k}f^{ij}\frac{\p \xi_{k}}{\p x_i}\frac{\p\log \mathbb D}{\p \xi_{k}}\frac{\p \log \mathbb  F}{\p x_{i}}=\sum_{i}\frac{\p\log \mathbb D}{\p \xi_{i}}\frac{\p \log \mathbb F}{\p x_{i}},
\end{equation*}
we have
$$
\left|\sum_{i,j}f^{ij} (\log \mathbb D)_{i}(\log \mathbb F)_{j}\right|\leq C\sum |(\log \mathbb  F)_{j}|.
$$
By $\left|\frac{\p f}{\p x_{i}}\right|=|\xi_{i}|\leq diam(\Delta)$,\;$\frac{\sum x_{k}^2}{(d+f)^2}\leq b$ and \eqref{eqn 2.10} we have, at $p^{*},$
\begin{equation}\label{eqn 2.12}
\left|\sum_{i,j}f^{ij} (\log \mathbb D)_{i}(\log \mathbb F)_{j}\right|\leq C \sum |(\log \mathbb F)_{j}|\leq C_{3}.
\end{equation}
where $C_{3}$ is the constant depending only on $b,diam(\Delta)$ and $n.$
Inserting \eqref{eqn 2.12} into \eqref{eqn3.4}, we obtain
\begin{align}\label{eqn3.6}\begin{split}&\varepsilon\frac{1+\sum x_k^2}{(d + f)^2}
\bigg[\frac{2\sum f^{ii}}{1+\sum x_k^2}-\frac{4\langle\nabla \sum x_k^2,\nabla f\rangle}{(1+\sum x_k^2)(d+f)}
-\frac{2n}{d + f} + \frac{6\|\nabla f\|^2}{(d + f)^2}\bigg]\\
&\quad
+  \frac{2n\|\nabla f\|^2}{(d+f)^2}- \frac{2n^2}{d+f}
+A- C_2n-C_{3}\leq 0.\end{split}\end{align}
By the Schwarz inequality, we have
\begin{equation*}
\left|\frac{4\langle\nabla \sum x_k^2,\nabla f\rangle}{(1+\sum x_k^2)(d+f)}\right|
\leq  \frac{\|\nabla\sum x_k^2\|^2}{4(1+\sum x_k^2)^2} + \frac{16\|\nabla f\|^2}{(d + f)^2}.\end{equation*}
Hence,
\begin{equation}\label{eqn3.7}
\left|\frac{4\langle\nabla \sum x_k^2,\nabla f\rangle}{(1+\sum x_k^2)(d+f)}\right|
\leq  \frac{\sum f^{ii}}{(1+\sum x_k^2)^2} + \frac{16\|\nabla f\|^2}{(d + f)^2}.\end{equation}
Combining \eqref{eqn3.6} and \eqref{eqn3.7} yields
\begin{align*}&\varepsilon\frac{1+\sum x_k^2}{(d + f)^2}
\bigg[\frac{\sum f^{ii}}{1+\sum x_k^2}
-\frac{2n}{d + f} - \frac{10\|\nabla f\|^2}{(d + f)^2}\bigg]\\
&\quad
+ \frac{2n\|\nabla f\|^2}{(d+f)^2}- \frac{2n^2}{d+f}
+A- C_2n- C_3\leq 0.\end{align*}
By choosing $\varepsilon>0$ such that $10\varepsilon b\leq 1$, we have
$$\varepsilon\frac{\sum f^{ii}}{(d + f)^2}+ A- C_4\leq 0.$$
By the relation between the geometric mean and the arithmetic mean, we get
$$\frac{\det(u_{ij})}{\left(d+f\right)^{ {2n} }}
=\frac{(\det(f^{ij}))^{-1}}{\left(d+f\right)^{  2n  }}\leq \frac{C(n)(\sum f^{ii})^{n}}{\left(d+f\right)^{  2n  }}\leq C_5.$$
Therefore, $\mathcal{F}(p^*)\le C_6$, and hence $\mathcal{F}\le C_6$ everywhere. The definition of $\mathcal{F}$ and the bound of $\mathbb D$ implies
$$\exp\left\{ -C_2 f \right\}\frac{\det(u_{ij})}{\left(d+f\right)^{ {2n} }}\leq C_7.$$
%We obtain
%$$\exp\left\{ -\mff C_3 f \right\}\frac{\det (D^2u)}{\left(d+f\right)^{2n}}\leq \mff C_4.$$
This is the desired estimate. \end{proof}

\section{ Convergence theorems in section}\label{Sec-Convergence}
\v
Let $\Omega^*\subset \mathbb{R}^n$. Denote by $\mathcal{F}(\Omega^*,C)$ the class of smooth convex functions
defined on $\Omega^*$ such that
$$ \inf_{\Omega^*} {u}=0,\;\;\;
u= C>0\;\;on\;\;\partial \Omega^*.$$
\begin{lemma}\label{lemma 4.1}
Let $\Omega^*\subset \mathbb{R}^n$ be a normalized domain, $u\in \mc F(\Omega^*, C)$ be a function satisfying the generalized Abreu Equation \eqref{eqn 1.1}. Suppose that there is a constant $C_1>0$ such that in $\Omega^*$ \begin{equation}\label{eqn_A}C_1\inv\leq
\det(u_{ij})\leq C_1.\end{equation}
Then for any $\Omega^\circ\subset\subset \Omega^* $, $p>2$, we have the
estimate \begin{equation}\|u\|_{W^{4, p}(\Omega^\circ)}\leq C ,\;\;\;\;\|u\|_{C^{3,\alpha}(\Omega^\circ)}\leq C,
\end{equation} where $C$ depends on $n, p, \mathbb{D}, C_1, \|A\|_{L^{\infty}(\Delta)}$, $dist(\Omega^\circ,
\partial \Omega^* )$.
\end{lemma}
\v\n
{\bf Proof.} In \cite{C-G} Caffarelli-Gutierrez  proved a H\"older estimate of
$\det(u_{ij})$ for homogeneous linearized Monge-Amp\`ere { equations}
assuming that the Monge-Amp\`ere measure $\mu[u]$ satisfies some
condition, which is guaranteed by \eqref{eqn_A}. { Consider the
generalized Abreu Equation
$$ \sum U^{ij}\mathbb F_{ij}= -A\mathbb D
,\;\;\;\mathbb{F}:=\frac{\mathbb D}{\det(u_{ij})} $$ where $A\in
L^\infty(\Omega).$ Since $\mathbb{D}\in C^{\infty}(\bar{\Delta})$ and $\mathbb{D}>0$, by the same argument in \cite{C-G}  one can obtain the   H\"older
continuity of $\det(u_{ij})$. Then Caffarelli's $C^{2,\alpha}$
estimates for Monge-Amp\`ere  { equations \cite{C1}} give  us
$$\|u\|_{C^{2,\alpha}(\Omega^*)}\leq C_2.$$
Following from the standard elliptic regularity theory we have
$\|u\|_{W^{4, p}(\Omega^*)}\leq C_3 $. By the Sobolev embedding theorem
\[\|u\|_{C^{3,\alpha}(\Omega^\circ)}\leq C_4  \|u\|_{W^{4, p}(\Omega^*)}.
\]
Then  the lemma follows.
$\blacksquare$
\v
Let $\Omega\subset \mathbb{R}^n$. Denote by $\mathcal{F}(\Omega,C)$ the class of smooth convex functions
defined on $\Omega$ such that
$$ \inf_{\Omega} {f}=0,\;\;\;
f= C>0\;\;on\;\;\partial \Omega.$$ Next we prove
the following convergence theorem.
\begin{theorem}\label{theorem_4.2}
Let $\Omega\subset \mathbb R^{n}$ be a  normalized domain. Let $f_{(k)}\in \mc
F(\Omega,C)$ be a sequence of functions satisfying the equation \begin{equation}\label{equ 4.4}
-\sum_{i,j} f_{(k)}^{ij}\frac{\p^2(\log \mathbb F_{(k)})}{\p x_i \p x_j} -\sum_{i,j}f_{(k)}^{ij}\frac{\p (\log \mathbb F_{(k)})}{\p x_i}\frac{\p (\log \mathbb D)}{\p x_j}=A_{(k)}.
\end{equation}
Suppose that $A_{(k)}$ $C^{m}$-converges to $A$  on $\bar\Omega$ and there are constants $0< C_1<C_2$ independent of $k$ such that
\begin{equation} \label{equ 4.5}
C_1\leq det\left(\frac{\p^2 f_{(k)}}{\p x_i\p x_j}\right)\leq C_2\end{equation}
hold in $\Omega$. Then there exists a subsequence
of functions, without loss of generality, still denoted by
$f_{(k)}$, { locally uniformly } { converging} to a function $f_\infty$ { in $\Omega$} and, for any open set $\Omega_o$ with $\bar{\Omega}_o\subset \Omega$, and for any $\alpha\in (0,1)$,  $f_{(k)}$ $C^{m+3,\alpha}$-converges
to $f_\infty$ in $\Omega_o$.

\v
\end{theorem}
\v\n
{\bf Proof.} It is obvious that there exists a subsequence
of functions, locally uniformly converging to a function $f_\infty$ in $\Omega$. A fundamental result on Monge-Amp\`ere equation tell us that $f_\infty$ is $C^{1,\alpha}$ and strictly convex (see \cite{Gui}).
Suppose that $f_\infty(p)=0$ for some point $p\in \Omega$. We choose the coordinates $x=(x_1,...,x_n)$ such that $x(p)=0$.
Put
$$u_{(k)}=\sum x_i\frac{\p f_{(k)}}{\p x_i} - f_{(k)},\;\;\;\Omega^*_{(k)}=\nabla f_{(k)}(\Omega),$$
$$u_{\infty}=\sum x_i\frac{\p f_{\infty}}{\p x_i} - f_{\infty},\;\;\;\Omega^*_{\infty}=\nabla f_{\infty}(\Omega).$$
We have $f_{\infty}(0)=u_{\infty}(0)=0$.
The key point of the proof of the Theorem is the following
\v\n
{\bf Claim.} There are constants $C>0$, $b> r >0$ such that $\bar{S}_{u_{\infty}}(0,C)$ is compact and
$$D_r(0)\subset S_{u_{\infty}}(0,C)\subset D_b(0).$$
\v\n

$$
x_{n+1}=f(x)
$$

{\bf Proof of Claim.} Denote
$$M:=\{(x, f_{\infty}(x))| x\in \Omega\},\;\;\;M^*:=\{(\xi, u_{\infty}(\xi))| \xi\in \Omega^*_{\infty}\}.$$
Then $M$ is a $C^{1,\alpha}$ strictly convex hypersurface with the support hyperplane $H=\{x | x_{n+1}=0\}$ at $0$. We look at the geometry meaning of $u_{\infty}$. Let $q\in \Omega$ be a point near $0$. The support hyperplane of $M$ at $(q,f_{\infty}(q))$ is given by
$$H_{(q,f_{\infty}(q))}=\left\{ (x_1,...,x_n,x_{n+1})| \sum x_i\frac{\p f}{\p x_i}(q) + x_{n+1}=\sum x_i(q)\frac{\p f}{\p x_i}(q) + f(q)\right\}.$$
The intersection
\begin{equation} \label{equ 4.2}
H_{(q,f_{\infty}(q))}\bigcap \{(0,...,0,x_{n+1}\}=(0,...,0, -u_{\infty}(q)).\end{equation}
In particular, we have
\v
{\bf ($\star$)} $u_{\infty}$ is monotonically increase along every ray from $0$: $\{x_i=a_it, t\geq 0,  \;i=1,...,n\}$,
where $a_i$ are constants with $\sum a_i^2=1$.
\v
\begin{figure}[t]
\includegraphics[width=6cm]{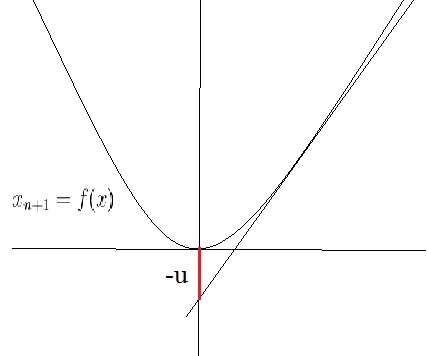}
\centering
\end{figure}

\v\n
By strictly convexity of $f_{\infty}$ we can find $b_1>b_2>0$, $d_1>d_2>0$ such that
\begin{itemize}
\item[(1)] $\bar{S}_{f_{\infty}}(0,b_2)\subset \bar{S}_{f_{\infty}}(0,b_1)\subset D_{d_1}(0)\subset \Omega$,
\item[(2)] $dist\left(\bar{S}_{f_{\infty}}(0,b_2), \p \bar{S}_{f_{\infty}}(0,b_1)\right)\geq d_2$.
\end{itemize}
Then
$$|\nabla f_{\infty}|\leq \frac{b_1}{d_1}\;\;\;\;\forall x\in \bar{S}_{f_{\infty}}(0,b_2).$$
It follows that $\bar{S}_{u_{\infty}}(0,C)\subset D_b(0)$ for some constant $b>0$. By compactness we can find $p\in \p S_{f_{\infty}}(0,b_2)$ such that $u_{\infty}(p)=min_{\p S_{f_{\infty}}(0,b_2)}\{u_{\infty}\}.$
By {\bf ($\star$)} we can find a set $\Omega^\circ\subset \bar{S}_{f_{\infty}}(0,b_2)\}$ such that
$$u_{\infty}(x)=u_{\infty}(p)\;\;\;\forall x\in \p \Omega^\circ.$$
Let $q\in \p\Omega^\circ$ be the point with $f_{\infty}(q)=min_{\p\Omega^\circ}\{f_{\infty}\}$.
By the strictly convexity of $f_{\infty}$, we have $f_{\infty}(q)>0$. Consider the convex cone $V$ with vertex $(0,0)$ and the base
$$\{(x_1,x_2,...,x_n, f_{\infty}(q))|x_1^2 + ... + x_n^2=d_1^2\}.$$
By the comparison theorem of normal maps, there exists a Euclidean ball $D_{r}(0)$ such that $D_{r}(0)\subset \nabla f_{\infty}(\bar{S}_{f_{\infty}}(0, f_{\infty}(q))$. We choose $C=u_{\infty}(p)$. Then
$$D_r(0)\subset S_{u_{\infty}}(0,C)\subset D_b(0).$$
The claim follows.

\v\n
By the claim we conclude that
$$\bar{S}_{u_{(k)}}(0,C/2):=\{\xi |u_{(k)}\leq C/2\}$$
is compact and contain a Euclidean ball for $k$ large enough. By \eqref{equ 4.5} we have
\begin{equation} \label{equ 4.3}
\frac{1}{C_2}\leq det\left(\frac{\p^2 u_{(k)}}{\p \xi_i\p \xi_j}\right)\leq \frac{1}{C_1}\end{equation}
A direct calculation shows that  $u_{(k)}$ satisfy the generalized Abreu Equation \eqref{eqn 1.1}. By Lemma \ref{lemma 4.1} $u_{(k)}$ $C^{m+3}$-converges to $u_{\infty}$. It follows that $f_{(k)}$ $C^{m+3}$-converges to $f_{\infty}$ in a neighborhood of $0$.
\v
Now let $p\in \Omega$ be an arbitrary point, let $l$ be the linear function defining the support hyperplane of $M$ at $(p,f_{\infty}(p))$. Let
$$\tilde{f}_{\infty}= f_{\infty}- l.$$
We use $\tilde{f}_{\infty}$ instead $f$ and use the same argument above. The theorem follows. $\blacksquare$

\section{Proof of the Main Theorem}\label{Sec-ProofMain}

Since $(\Delta,\mathbb{D},A)$ is uniformly $K$-stable and $A^{(k)}$ converges to $A$ smoothly in $\bar \Delta$,
then $(\Delta, \mathbb{D},A^{(k)})$ is uniformly $K$-stable for large $k$, i.e.,
$\Delta$ is $(A^{(k)},\mathbb{D},\lambda)$-stable
for some constant $\lambda>0$ independent of $k$.
Since $u^{(k)}$ satisfies the generalized Abreu Equation \eqref{eqn 1.1}, then
$$\mc L_{A_k}(u^{(k)})=\int_{\Delta}\sum_{i,j} (u^{(k)})^{ij}(u^{(k)})_{ij}\mathbb{D}d\mu=n\int_{\Delta}\mathbb{D}d\mu
$$ and hence,
$$\int_{\partial \Delta} u d \sigma\leq n\lambda^{-1}\frac{max_{\Delta}\mathbb{D}}{min_{\Delta}\mathbb{D}} \mbox{Area}(\Delta).
$$
It follows that $u^{(k)}$ locally and
uniformly converges to a convex function $u$ in $\Delta$. %(cf. \cite{D1})
\v

 {\it Claim.} For any point $\xi\in \Delta$ and any
$B_{\delta}(\xi)\subset \Delta$, there exists a point
$\xi_o\in B_{\delta}(\xi)$ such that $u$ has second derivatives
and is strictly convex at $\xi_o$. Here, $B_{\delta}(\xi)$ denotes the
Euclidean ball centered at $\xi$ with radius $\delta.$

\v

The proof of the claim is the same as in \cite{CLS5}, see also \cite{CHLS}.

\v
We now choose coordinates such that $\xi_o=0$.
By adding linear functions, we assume that all $u^{(k)}$ and $u$ are normalized at $0$.
%({\bf Why do we still have convergence after adjusting by linear functions?})
Since $u$ is strictly convex at $0$, there exist constants  $\epsilon'>0$, $d_2>d_1>0$ and $b' >0$,
independent of $k$, such that, for large $k$,
$$B_{d_1}(0)\subset\bar{S}_{u^{(k)}}(0,\epsilon')\subset B_{d_2}(0)\subset \Delta,$$ and
$$
\sum_i \left(\frac{\partial u^{(k)}}{\partial \xi_i}\right)^2\leq b'
\quad\text{in } S_{u^{(k)}}(0,\epsilon').$$  By
Lemma \ref{lemma_3.1} and Lemma \ref{lemma_3.3}, we have
%uniform estimates for $\det(D^2u^{(k)})$
\begin{equation}\label{eqn 4.3}
C_1\leq \det(u_{ij}^{(k)}) \leq C_2
\quad\text{in }S_{u^{(k)}}(0,\frac{1}{2}\epsilon'),\end{equation}
where $C_1<C_2$ are positive constants independent of $k$.
By Lemma \ref{lemma 4.1} $\{u^{(k)}\}$ converges smoothly to $u$.
Therefore, $u$ is  a smooth and
strictly convex function in $S_u(0,\epsilon'/2)$.
\v
Let $f^{(k)}$ be the Legendre transform of $u^{(k)}$.
Then, $\{f^{(k)}\}$ locally
uniformly converges to a convex function $f$ defined in the
whole $\mathbb{R}^n$. Furthermore, in a neighborhood of $0$,
$f$ is a smooth and strictly convex function such that its
Legrendre transform $u$ satisfies the generalized Abreu Equation \eqref{eqn 1.1}.
By the convexity of $f^{(k)}$ and the local and uniform convergence of
$\{f^{(k)}\}$ to $f$, we conclude, for any $k$,
$$\frac{1+\sum_i x_i^2}{(d+ f^{(k)})^2} \leq b\quad\text{in }\mathbb R^n,$$
%({\bf Where does it hold?})
and, for any $C>1$,
$$B_r(0)\subset S_{f^{(k)}}(0,C)\subset B_{R_C}(0),$$
for some positive  constants $d$, $b$, $r$ and $R_C=R(C)>0$. By Lemma \ref{lemma_3.1} and Lemma \ref{lemma_3.5}, we have
$$\exp\{-\mff C_3 C\}\frac{1}{(d+C)^{2n}}\leq \det (f_{ij}^{(k)})\leq \mff C_1 .$$
We note that each $f^{(k)}$ satisfies \eqref{equ 4.4}. By Theorem \ref{theorem_4.2}
we conclude that $\{f^{(k)}\}$ uniformly and smoothly converges to $f$
in $S_f(0,C/2)$. Since $C$ is arbitrary, $f$ is a smooth and strictly convex function in $\mathbb{R}^n$,
and the sequence $\{f^{(k)}\}$ locally and smoothly converges to $f$. By Legendre transforms,
we obtain that $u$ is a smooth and strictly convex function in $\Delta$
and that the sequence $\{u^{(k)}\}$ locally and smoothly converges to $u$.
This completes the proof of Theorem \ref{theorem_1.2}.
\v\v\v\v

\bibliography{<your-bib-database>}

\end{document}